\documentclass[reqno]{amsart}
\usepackage[english]{babel}
\usepackage{amssymb}
\usepackage[T1]{fontenc}
\newtheorem{theorem}{Theorem}[section]
\newtheorem{proposition}[theorem]{Proposition}
\newtheorem{lemma}[theorem]{Lemma}
\newtheorem{corollary}[theorem]{Corollary}
\theoremstyle{definition}
\newtheorem{definition}[theorem]{Definition}
\newtheorem{example}[theorem]{Example}
\theoremstyle{remark}
\newtheorem*{remark}{Remark}
\numberwithin{equation}{section}

\newcommand{\RE}{\mbox{$\mathbb{R}$}}

\newcommand{\C}{\mbox{$\mathbb{C}$}}

\begin{document}

\title{The Dirichlet problem for $m$-subharmonic functions on compact sets}

\author{Per \AA hag}\address{Department of Mathematics and Mathematical Statistics\\ Ume\aa \ University\\SE-901~87 Ume\aa \\ Sweden}\email{Per.Ahag@umu.se}

\author{Rafa\l\ Czy{\.z}}\address{Faculty of Mathematics and Computer Science, Jagiellonian University, \L ojasiewicza~6, 30-348 Krak\'ow, Poland}
\thanks{The second-named author was partially supported by NCN grant DEC-2013/08/A/ST1/00312.}
\email{Rafal.Czyz@im.uj.edu.pl}

\author{Lisa Hed}\address{Department of Mathematics and Mathematical Statistics\\ Ume\aa \ University\\SE-901~87 Ume\aa \\ Sweden}
\email{Lisa.Hed@umu.se}


\begin{abstract}
 We characterize those compact sets for which the Dirichlet problem has a solution within the class of continuous $m$-subharmonic functions defined on a compact set, and then within the class of $m$-harmonic functions.
\end{abstract}

\maketitle

\begin{center}\bf
\today
\end{center}

\section{Introduction}

 A fundamental tool in the study of uniform algebras is the class of subharmonic functions defined on compact sets, and its dual, the
 Jensen measures. In~\cite{Gamelin}, Gamelin presented a model that can be used both for subharmonic as well as plurisubharmonic functions defined on compact sets. In this note we shall use this model to investigate the Dirichlet problem for $m$-(sub)harmonic functions. Our inspiration is the work of Poletsky~\cite{Poletsky1} and especially  Poletsky-Sigurdsson~\cite{PoletskySigurdsson1,PoletskySigurdsson2}.

  Two natural types of boundaries in potential theory are the Choquet boundary (Definition~\ref{def_Jzm}) with respect to a given class of Jensen measures, and the \v{S}ilov boundary (Definition~\ref{def_BXm}). In our study of the Dirichlet problem these boundaries have a prominent role, and therefore we shall in Section~\ref{sec_boundary} put extra attention on them in terms of for example peak points and harmonic $m$-measures. We shall then, in Section~\ref{sec_DPcont}, characterize those compact sets for which the Dirichlet problem has a
 solution within the class of continuous $m$-subharmonic functions (Theorem~\ref{thm_oregular} and Theorem~\ref{thm_whole}). We end this note in Section~\ref{sec_DPharmonic} with a Dirichlet problem for $m$-harmonic functions (Theorem~\ref{thm_poisson}). In the $1$-subharmonic case, these results were obtained by Hansen among others (see e.g.~\cite{BlietnerHansen,Hansen,perkins} and the references therein), and in the $n$-subharmonic case these results were proved by Poletsky-Sigurdsson~\cite{PoletskySigurdsson1,PoletskySigurdsson2}. In this note we prove the $1<m<n$ cases. We start in Section~\ref{sec_jensen} by stating some basic definitions and facts.

\section{Jensen measures and envelopes}\label{sec_jensen}

Let $\mathcal{SH}_m^o(X)$ denote the set of functions that are the restriction to $X$ of functions that are $m$-subharmonic and continuous on some neighborhood of $X\subseteq\C^n$. Furthermore, let $\mathcal{USC}(X)$ be the set of upper semicontinuous functions defined on $X$. For a background on $m$-subharmonic functions defined on an open set see e.g.~\cite{AS,L}. Recall that
\[
\mathcal{PSH}(\Omega)=\mathcal{SH}_n (\Omega)\subset\cdots\subset \mathcal{SH}_1(\Omega)=\mathcal{SH}(\Omega)\, ,
\]
where $\Omega$ is an open domain in $\C^n$, $\mathcal{PSH}$ denotes the plurisubharmonic functions, $\mathcal{SH}$ denotes the subharmonic functions and $\mathcal{SH}_m$ denotes the $m$-subharmonic functions defined on $\Omega$.

Next, we define a class of Jensen measures.

\begin{definition}\label{def_JzmK} Let $X$ be a compact set in $\C^n$, $1\leq m\leq n$, and let $\mu$ be a non-negative regular Borel measure defined on $X$ with $\mu(X)=1$. We say that $\mu$ is a \emph{Jensen measure with barycenter} $z\in X$ \emph{w.r.t.} $\mathcal{SH}_m^o(X)$ if
  \[
  u(z)\leq \int_{X} u \, d\mu \qquad\qquad \text{for all } u  \in \mathcal{SH}_m^o(X)\, .
  \]
The set of such measures will be denoted by $\mathcal{J}_z^m(X)$.
\end{definition}

 With the help of the Jensen measures defined in Definition~\ref{def_JzmK} we can
 now define $m$-subharmonic functions defined on compact sets. For more results about these functions, see~\cite{ACH}.

\begin{definition} \label{def_msubkomp}
Let $X$ be a compact set in $\C^n$. An upper semicontinuous function $u$ defined on $X$ is said to be \emph{$m$-subharmonic on $X$}, $1\leq m\leq n$,  if
\[
u(z) \leq \int_X u \, d\mu\, , \  \text { for all } \ z \in X \  \text { and all }\  \mu \in \mathcal{J}_z^m(X)\, .
\]
The set of  $m$-subharmonic functions defined on $X$ will be denoted by $\mathcal{SH}_m(X)$. A function $h:X\to\RE$ is called \emph{$m$-harmonic} if $h$, and $-h$, are $m$-subharmonic. The set of all $m$-harmonic functions defined on $X$ will be denoted by $\mathcal{H}_m(X)$. We shall call $n$-harmonic functions \emph{pluriharmonic}, and denote it by $\mathcal {PH}(X)=\mathcal{H}_n(X)$.
\end{definition}
\begin{remark}
By definition, we see that $\mathcal{SH}_m^o(X) \subseteq \mathcal{SH}_m(X)$.
\end{remark}
\begin{remark}
It follows from Definition~\ref{def_JzmK} that for any $z\in X$ we have
\[
\mathcal J_z^m(X)=\bigcap _{X\subset U}\mathcal J_z^m(U)=\bigcap _{X\subset U}\mathcal J_z^{m,c}(U)\, ,
\]
where $U$ denotes an \emph{open} set in $\mathbb C^n$. Recall that $\mu\in \mathcal J_z^m(U)$ $(\mathcal J_z^{m,c}(U))$ if $\mu$ is a probability measure with compact support in $U$, and for all $u\in \mathcal {SH}_m(U)$  $(u\in \mathcal {SH}_m(U)\cap \mathcal C(U))$ it holds that
\[
u(z)\leq \int u\,d\mu\, .
\]
Thanks to Theorem~2.2 in~\cite{ACH2} we have that $\mathcal J_z^m(U)=\mathcal J_z^{m,c}(U)$.
\end{remark}
\begin{remark}
Sometimes we will leave out the parenthesis and only write $\mathcal{J}_{z}^{m}$ instead of $\mathcal{J}_{z}^{m}(X)$, where $X$ is a compact set.
\end{remark}
\begin{remark}
From Definition~\ref{def_msubkomp} it follows that a continuous function $h:X\to\RE$ is $m$-harmonic if, and only if, for every $z_0\in\mathcal{J}_{z_0}^{m}$ we have that
\[
h(z_0)=\int h\, d\mu\, .
\]
\end{remark}
\begin{remark}
For Borel probability measures let us define the following two classes
\begin{align*}
\mathcal{M}_1 &=\left\{\mu : u(z)\leq \int_X u \, d\mu \text{ for all } u\in \mathcal {SH}_m(X)\right\}\\[2mm]
\mathcal{M}_2 &=\left\{\mu : u(z)\leq \int_X u \, d\mu \text{ for all } u\in \mathcal {SH}_m(X)\cap\mathcal C(X)\right\}\, .
\end{align*}
It follows from the proof of Theorem 2.8 in~\cite{ACH} that
\[
\mathcal{M}_1=\mathcal{M}_2=\mathcal{J}_z^m(X)\, .
\]
This means that the class of Jensen measures can be generated by the class of $m$-subharmonic functions on $X$ or by the class of \emph{continuous} $m$-subharmonic functions on $X$.
\end{remark}

In Definition~\ref{def} we introduce two useful envelope constructions.

\begin{definition}\label{def}
Assume that $X \subseteq \C^n$ is a compact set, and $1\leq m\leq n$.
For $f\in \mathcal C(X)$ we define
\[
\textbf {S}_f(z)=\sup\left\{v(z): v\in \mathcal {SH}_m(X), v\leq f\right\}\, ,
\]
and similarly
\[
\textbf {S}^c_f(z)=\sup\left\{v(z): v\in \mathcal {SH}_m(X)\cap\mathcal C(X), v\leq f\right\}\, .
\]
\end{definition}

We shall need the following version of Edwards' celebrated duality theorem (see Theorem 2.8 in~\cite{ACH}).

\begin{theorem}\label{thm_edwards}
Let $X$ be a compact subset in $\mathbb C^n$, $1\leq m\leq n$, and let $f$ be a real-valued lower semicontinuous function defined
on $X$. Then we have that
\begin{enumerate}\itemsep2mm
 \item[$(a)$]
\[
\sup\left \{\psi(z):\psi \in \mathcal{SH}_m^o(X) , \psi \leq f \right\}=\inf\left \{\int f \, d\mu : \mu \in \mathcal{J}_z^m(X)\right\}\, ,
\ \text {and}
\]
\item[$(b)$]
\[
\textbf {S}_{f}(z)=\textbf {S}^c_{f}(z)=\inf\left\{\int f \, d\mu: \mu \in \mathcal{J}_z^m(X)\right\}\, .
\]
\end{enumerate}
\end{theorem}

In Theorem~\ref{peak} we shall use the following lemma.

\begin{lemma}\label{lemma_approx}
Let $X$ be a compact subset in $\mathbb C^n$, $1\leq m\leq n$, and let $f$ be a real-valued lower semicontinuous function defined
on $X$. Then there exists a sequence $u_j\in \mathcal {SH}_m^o(X)$ such that $u_j\nearrow \textbf{S}_f$, as $j\to \infty$.
\end{lemma}
\begin{proof}
The proof follows from Theorem~\ref{thm_edwards}, and Choquet's lemma (see e.g. \\ Lemma~2.3.4 in~\cite{klimek}).
\end{proof}

\section{The Choquet and \v{S}ilov boundaries of compact sets}\label{sec_boundary}

The Choquet boundary (Definition~\ref{def_Jzm}) of a compact set $X$ w.r.t. $\mathcal{J}_{z_0}^m$ and its topological closure
the \v{S}ilov boundary (Definition~\ref{def_BXm}) are central concepts in the Dirichlet problems studied in Section~\ref{sec_DPcont} and Section~\ref{sec_DPharmonic}, so in this section we shall characterize these boundaries in terms
of peak points and  $m$-harmonic measures.

\begin{definition}\label{def_Jzm} Let $1\leq m\leq n$, and let $X$ be a compact set in $\C^n$.
The \emph{Choquet boundary} of $X$ w.r.t. $\mathcal{J}_{z_0}^m$ is defined as
\[
O_X^m=\big\{z\in X: \mathcal{J}_{z_0}^m=\{\delta_z\} \big\}\, .
\]
\end{definition}

From Lemma~1.10 in~\cite{Gamelin} it follows that $O_X^m$ is a $G_{\delta}$-set. Let $\Omega$ be a bounded domain in $\mathbb C^n$, and set $X=\bar \Omega$. Then the Choquet boundary is contained in the topological boundary, i.e.
$O_X^m\subseteq \partial X$.

Next we introduce the concept of $m$-subharmonic peak points.

\begin{definition}\label{def_peakpoint}
Let $1\leq m\leq n$, and let $X$ be a compact set in $\C^n$. We say that a point $z\in X$ is a \emph{m-subharmonic peak point} (or simply a \emph{peak point}) if there exists a function $u\in\mathcal{SH}_m(X)$ such that $u(z)=0$, and $u(w)<0$ for $w\in X\backslash\{z\}$. The function $u$ is then called a \emph{peak function}.
\end{definition}

Using Gamelin's more general setting we can, from Theorem~1.13 in~\cite{Gamelin}, draw the conclusion that: A point $z\in X$ is a $m$-subharmonic peak point if, and only if, there exists a function $u \in \mathcal{SH}_m(X)\cap \mathcal{C}(X)$ such that $u(z)=0$ and $u(w)<0$ for $w\in X\backslash\{z\}$.

We shall later use Lemma~\ref{lem_choquet} in Theorem~\ref{thm_oregular}, and Lemma~\ref{lemma_peak2} is used in Theorem~\ref{peak}.

\begin{lemma}\label{lem_choquet}
Let $1\leq m\leq n$, and let $X$ be a compact set in $\C^n$. Then $z\in O_X^m$ if, and only if, for every $f\in \mathcal{C}(X)$ we have that
$\textbf{S}_f(z)=f(z)$.
\end{lemma}
\begin{proof} Cf. page~10 in~\cite{Gamelin}.
\end{proof}

\begin{lemma}\label{lemma_peak2}
Let $1\leq m\leq n$, and let $X$ be a compact set in  $\C^n$. A point $z\in X$ is a $m$-subharmonic peak point if, and only if,
for any neighborhood $V$ of $z$ there exists $u\in \mathcal {SH}_m(X)$ such that $u<0$, $u(z)=-1$ and $u\leq -2$ on $X\setminus V$.
\end{lemma}
\begin{proof} The implication $\Rightarrow$ is immediate. To prove the converse implication take $z_0\in X$, and let $\epsilon_j\searrow 0$, as $j\to \infty$, be a  sequence such that
\[
\epsilon_j<\frac 12\frac {\left(\frac 34\right)^j}{1-\left (\frac {3}{4}\right)^j}\, .
\]
Let $Y_n\subseteq X$ be closed subsets such that
\[
X\setminus \{z_0\}=\bigcup _{n=1}^{\infty}Y_n\, .
\]
Now we shall define a sequence of functions from $\mathcal {SH}_m(X)$. Let $u_0=-1$. Suppose that we already have chosen $u_0,\dots,u_j\in \mathcal {SH}_m(X)\cap \mathcal C(X)$, such that for $1\leq k\leq j$ we have that $u_k<0$, $u_k(z_0)=-1$, and
\[
u_k\leq -2 \qquad  \text{on } U_{k-1}\cup Y_{k-1}\, ,
\]
where
\[
U_{k-1}=\left\{w\in X: \max_{1\leq l\leq k-1} u_l(w)\geq -1+\epsilon_{k-1} \right\}\, .
\]
Note that $\{U_k\}$ is an increasing sequence of closed sets. Now take a function $u_{j+1}\in \mathcal {SH}_m(X)$ such that $u_{j+1}\leq -2$ on $Y_j\cup U_j$,
and $u_{j+1}(z_0)=-1$. Let us then define
\[
u(z)=\frac 14\sum_{j=0}^{\infty}\left(\frac 34\right)^ju_j\, .
\]
This construction then implies that $u\in \mathcal {SH}_m(X)$ and $u(z_0)=-1$. Now suppose that $w\neq z_0$ and $w\notin \bigcup_{k=1}^{\infty} U_k$, that $u_j(w)\leq -1$ for all $j$, and $u_j(w)\leq -2$ for at least one $j$, therefore $u(w)<-1$. Assume next that $w\in \bigcup_{k=1}^{\infty} U_k$. If $w\in U_l\setminus U_{l-1}$, then
\[
\begin{aligned}
u_j(w)&\leq -1+\epsilon_{l-1}, \ \text{for} \ 1\leq j\leq l-1,\\
u_l(w)&<0, \\
u_j(w)&\leq -2, \ \text{for} \ j>l.
\end{aligned}
\]
Now we have that
\[
\begin{aligned}
u(w)&\leq \frac 14 \left((-1+\epsilon_l)\sum_{j=0}^{l-1}\left(\frac 34\right)^j-2\sum_{j=l+1}^{\infty}\left(\frac 34\right)^j\right)\\
&=-1+\epsilon_{l}\left(1-\left(\frac 34\right)^l\right)-\frac 12\left(\frac {3}{4}\right)^l<-1.
\end{aligned}
\]
This means that $u+1$ is a peak function.
\end{proof}

In Theorem~\ref{peak} we characterize the Choquet boundary of $X$ w.r.t. $\mathcal{J}_{z_0}^m$
in terms of peak points.

\begin{theorem}\label{peak}
Let $1\leq m\leq n$, and let $X$ be a compact set in  $\C^n$. Then $z\in O_X^m$ if, and only if, $z$ is a peak point.
\end{theorem}
\begin{proof}
If $z$ is a peak point, then there exists a function $u\in \mathcal {SH}_m(X)\cap \mathcal C(X)$ such that $u(z)=0$ and $u(w)<0$, for all $w\neq z$.
Let $\mu\in \mathcal J_z^m$. Then it holds that
\[
0=u(z)\leq \int_Xu\,d\mu\leq 0\, ,
\]
which implies that $\operatorname{supp}\mu\subseteq \{w: u(w)=0\}$. Hence, $\mu=\delta_z$.

On the other hand, let $z_0\in O_X^m$ and let $V$ be any neighborhood of $z_0$. Furthermore, let $f\in \mathcal C(X)$ be such that $f<0$, $f(z_0)=-1$, and $f<-4$ on $X\setminus V$. Then, $\textbf{S}_f(z_0)=-1$, and $\textbf{S}_f<-4$ on $X\setminus V$. From Lemma~\ref{lemma_approx} it follows that one can find a function $v\in \mathcal {SH}_m^o(X)$ such that $v\leq \textbf{S}_f$, and $-1>v(z_0)>-2$. Then the function $u$ defined by
\[
u(z)=-\frac {v(z)}{v(z_0)}
\]
satisfies $u(z_0)=-1$, and $u<-2$ on $X\setminus V$. The proof is finished by Lemma~\ref{lemma_peak2}.
\end{proof}

\begin{remark} It follows from Theorem~\ref{peak} that $O_X^m$ is non-empty.
\end{remark}

Next, we introduce the \v{S}ilov boundary of a compact set $X$ w.r.t. $\mathcal{J}_{z_0}^m$.

\begin{definition}\label{def_BXm} Let $1\leq m\leq n$, and let $X$ be a compact set in $\C^n$. The \emph{\v{S}ilov boundary, $B_X^m$, of $X$} is defined to be the topological closure of $O_X^m$.
\end{definition}

\begin{remark}
Obviously, it holds that $O_X^m\subseteq B^m_X$, and $O_X^m$ is closed if, and only if, $O_X^m=B^m_X$.
\end{remark}

It is not always true that a $m$-subharmonic function must attain its maximum on the potential boundary, $B^m_X$ (see the example before Theorem~4.3 in~\cite{PoletskySigurdsson1}). But we have at least the following weak maximum principle that we shall use in our study of the \v{S}ilov boundary and the $m$-harmonic measure.

\begin{theorem}\label{thm_weakmaxprinciple}
Let $1\leq m\leq n$, and let $X$ be a compact set in $\C^n$. If $u\in\mathcal{SH}_m(X)$, then
\[
u(z)\leq \sup_{w\in B^m_X} u(w)\qquad \text{ for all } z\in X\, .
\]
\end{theorem}
\begin{proof} See Theorem~1.12 in~\cite{Gamelin}.
\end{proof}

\begin{definition}\label{def_harmonic measure} Let $1\leq m\leq n$, and let $X$ be a compact set in $\C^n$. The \emph{$m$-harmonic measure} of a subset $E\subseteq X$ is defined as the function
\[
\omega(z,E,X)=\inf_{V\supset E\atop V open}\sup_{\mu\in \mathcal{J}_z^m} \mu (V)\, .
\]
\end{definition}

We have the following estimate.

\begin{theorem}\label{twoconstanttheorem}
Let $1\leq m\leq n$, $K$ and $k$ be constants and let $X$ be a compact set in $\C^n$. If $u\in \mathcal {SH}_m(X)$ satisfies $u\leq K$ on $X$, and $u\leq k<K$ on some set $Y\subseteq X$, then
\[
u(z)\leq k\omega(z,Y,X)+K(1-\omega(z,Y,X)),\qquad z\in X\, .
\]
\end{theorem}
\begin{proof}
Fix $\epsilon >0$, and set
\[
U=\{z\in X: u(z)<k+\epsilon\}\, .
\]
Then $U$ is an open set such that $Y\subseteq U$. For all $z\in X$, there exists a measure $\mu\in \mathcal J_z^m(X)$, and an open set $V$, such that
$Y\subseteq V\subseteq U$ and
\[
\omega(z,Y,X)\leq \mu(V)\leq \omega(z,Y,X)+\epsilon\, .
\]
Then we have that
\begin{multline*}
u(z)\leq \int_Xu\,d\mu=\int_Vu\,d\mu+\int_{X\setminus V}u\,d\mu\\ \leq (k+\epsilon)\big(\omega(z,Y,X)+\epsilon)+K(1-\omega(z,Y,X)\big)\, .
\end{multline*}
If we let $\epsilon \to 0^+$, then we get the desired conclusion.
\end{proof}

\begin{theorem}\label{thm_propBXm} Let $1\leq m\leq n$, and let $X$ be a compact set in $\C^n$. The \v{S}ilov boundary, $B^m_X$, is the smallest closed set $E$ such that $\omega(z,E,X)$ is identically 1.
\end{theorem}
\begin{proof}
First we shall prove that $\omega(z,B_X^m,X)=1$ for all $z\in X$. To prove this assume by contradiction that there exists a $z_0\in X$ such that $\omega(z_0,B_X^m,X)<1$. Then there exists an open neighborhood $V$ of $B_X^m$, and $0<c<1$, such that for all $\mu\in \mathcal J_{z_0}^m$ it holds that
\[
\mu(V)<c\, .
\]
Let $W$ be an open set such that $B_X^m\Subset W\Subset V$, and let $f\in \mathcal C(X)$ be such that $f=-1$ on $W$, and $f=0$ on $X\setminus V$. Then we have that
\begin{equation}\label{2}
\textbf{S}_f\leq f=-1 \ \text {on} \ W\, ,
\end{equation}
and thanks to Edwards' theorem (Theorem~\ref{thm_edwards}) we have that
\[
\textbf{S}_f(z_0)=\inf\left \{\int \phi \, d\mu : \mu \in \mathcal{J}_{z_0}^m(X)\right\}\, .
\]
For given $\mu\in \mathcal J_{z_0}^m$ it holds that
\[
\int_Xf\,d\mu=\int_{X\setminus V}f\,d\mu+\int_Vf\,d\mu>0+(-1)(1-c)=-1+c\, ,
\]
so
\begin{equation}\label{3}
\textbf{S}_f(z_0)>-1+c\, .
\end{equation}
Therefore, by~(\ref{2}) and~(\ref{3}) we conclude that there exists a function $u\in \mathcal {SH}_m(X)$ such that $u<-1$ on $W$, and $u(z_0)>-1$. But this is impossible since by Theorem~\ref{thm_weakmaxprinciple} each $m$-subharmonic function must attain its maximum on $B_X^m$. This ends the first part of the proof.

Next, assume that there exists a proper closed subset $E$ of $B_X^m$ such that for all $z\in X$ we have that $\omega(z,E,X)=1$. Then there exist a point $z_0\in  O_X^m\setminus E$, and a neighborhood $V$ of $E$ such that $z_0\notin V$. Then since $\mathcal J_{z_0}^m=\{\delta_{z_0}\}$, we get that $\omega(z_0,E,X)=0$ and a contradiction is obtained.
\end{proof}

\begin{corollary}\label{cor_propBXm} Let $1\leq m\leq n$, and let $X$ be a compact set in $\C^n$. The \v{S}ilov boundary, $B^m_X$, is the smallest closed set such that for every $z\in X$ there exists a Jensen measure $\mu\in \mathcal{J}_z^m$ such that $\operatorname{supp} \mu\subseteq  B^m_X$.
\end{corollary}

\begin{proof} Assume that $Y$ is a subset of $X$ such that for every $z\in X$ there
exists a Jensen measure $\mu\in \mathcal{J}_z^m$ such that $\operatorname{supp} \mu\subseteq Y$. For $z$ from the Choquet boundary we have that  $\mathcal{J}_z^m=\{\delta_z\}$. Therefore it follows that
$O^m_X\subseteq Y$, and hence $B^m_X\subseteq Y$.   For $z\in X$, and for any neighborhood $V$ of $Y$, we have that
\[
\sup_{\mu\in\mathcal{J}_z^m} \mu(V)=1\, ,
\]
and therefore $\omega(z,Y,X)=1$. If $Y$ is the smallest closed set with the assumed property it now follows by using Theorem~\ref{thm_propBXm} that $B^m_X=Y$.
\end{proof}

In Definition~\ref{def_JzBm} we introduce the subset, $\mathcal{J}_z^{b,m}$, of Jensen measures
$\mathcal{J}_z^m$ whose support is contained in the \v{S}ilov boundary, $B^m_X$. We
shall need  $\mathcal{J}_z^{b,m}$ in Lemma~\ref{lem_poisson}, Proposition~\ref{prop_PBf}, and in Theorem~\ref{thm_poisson}.

\begin{definition}\label{def_JzBm} Let $1\leq m\leq n$, $X$ be a compact set in $\C^n$ and $z \in X$. Then we define
\[
 \mathcal{J}_z^{b,m}=\{\mu\in \mathcal{J}_z^m: \operatorname{supp} \mu\subseteq  B^m_X \}\, .
\]
\end{definition}

A direct consequence of Corollary~\ref{cor_propBXm} is that $\mathcal{J}_z^{b,m}$ is non-empty.

\begin{corollary}\label{cor_JzBm} Let $1\leq m\leq n$, and let $X$ be a compact set in $\C^n$. For every $z\in X$ we have that $\mathcal{J}_z^{b,m}$ is non-empty.
\end{corollary}
\begin{proof} This follows from Corollary~\ref{cor_propBXm}.
\end{proof}

In solving the Dirichlet problem in the case when the Choquet boundary is the whole compact set
(Theorem~\ref{thm_whole}) we shall need $\mathcal I_X^m(z)$, defined below, together with Proposition~\ref{propositionIset}.
The inspiration behind $\mathcal I_X^m(z)$ is from potential theory, and it is explained in the remark after Proposition~\ref{propositionIset}.

\begin{definition}\label{definitionIset}
Let $1\leq m\leq n$, and let $X$ be a compact set in $\C^n$. For $z\in X$ let us define the following set
\[
\mathcal I_X^m(z)=\{w\in X: \omega(z,\bar{B}(w,r)\cap X,X)>0,  \ \text { for all } \ r>0\}.
\]
\end{definition}

\begin{proposition}\label{propositionIset}
Let $1\leq m\leq n$, and let $X$ be a compact set in $\C^n$. Then for $z\in X$ we have that
\begin{enumerate}\itemsep2mm
\item $\mathcal I_X^m(z)$ is a closed set;

\item if $\mu \in \mathcal J_z^m(X)$ then $\operatorname{supp}\mu\subseteq \mathcal I_X^m(z)$;

\item $\mathcal I_X^m(z)=\{z\}$ if, and only if, $z\in O_X^m$.
\end{enumerate}
\end{proposition}
\begin{proof} (1) First note that
\[
\mathcal I_X^m(z)=\bigcap_{r>0}Y_r, \ \text { where } \ Y_r=\{w\in X: \omega(z,\bar{B}(w,r)\cap X,X)>0\}\, ,
\]
and therefore it is sufficient to prove that the sets $Y_r$ are closed. Let $Y_r\ni x_j\to x\in X$. Then for every $j$ there exists an open set $V_j\supset \bar {B}(x_j,r)\cap X$, and $\mu_j\in \mathcal J_z^m(X)$ such that $\mu_j(V_j)>0$. From a compactness argument there exists a $j_0$ such that $V_{j_0}\supset \bar{B}(x,r)\cap X$, and therefore $\omega(z,\bar{B}(x,r)\cap X,X)>0$, so $x\in Y_r$.

(2) Fix $r>0$. Let $\mu\in \mathcal J_z^m(X)$, and let $w\in \operatorname {supp}\mu$, then $\mu(\bar{B}(w,r)\cap X)>0$. This means that $\omega(z,\bar{B}(w,r)\cap X,X)>0$, hence $w\in \mathcal I_X^m(z)$.

(3) It follows from (2) that if $\mathcal I_X^m(z)=\{z\}$, then $\mathcal J_z^m(X)=\{\delta_z\}$. Thus, $z\in O_X^m$. On the other hand, if $z\in O_X^m$, then for $w\neq z$ we have that $\omega(z,\bar{B}(w,r),X)=0$ if $r<\|z-w\|$. Therefore, $w\notin \mathcal I_X^m(z)$, which implies that $\mathcal I_X^m(z)=\{z\}$.
\end{proof}

\begin{remark}
There is a very nice characterization of those points for which $\mathcal I_X^1(z)=\{z\}$ in the case of subharmonic functions. Namely, $\mathcal I_X^1(z)=\{z\}$ if, and only if, $X^c$ is not thin at $z$ (see Theorem 3.3 in~\cite{Poletsky2}). A similar result for $m$-subharmonic functions, $m>1$, is not possible. To see this look at Example~\ref{ex2_poisson}: Then for all $z\in \partial \bar{\mathbb D}^n$ the set $\big(\bar{\mathbb D}^n\big)^c$ is not $m$-thin at $z$, but $\mathcal I_{\bar{\mathbb{D}}^n}^m(z)\neq\{z\}$, $m>1$, if e.g. $z\in \mathbb{D}\times\dots\times \mathbb {D}\times\partial \mathbb{D}$.
\end{remark}

\section{The Dirichlet problem for continuous $m$-subharmonic functions}\label{sec_DPcont}

In Theorem~\ref{thm_oregular}, we characterize those compact sets $X$ for which the Dirichlet problem has a solution within the class of continuous $m$-subharmonic functions defined on a compact set. To obtain this we need the notion of $O^m$-regular compact sets (Definition~\ref{def_Omregular}). We end this section with Theorem~\ref{thm_whole} where we consider the case when $X$ is equal to its Choquet boundary.

\begin{definition}\label{def_Omregular}
Let $1\leq m\leq n$. We say that a compact set $X$ in $\C^n$ is \emph{$O^m$-regular} if $O_X^m$ is a closed subset
of $X$.
\end{definition}

The next theorem provides the characterization of $O^m$-regular sets.

\begin{theorem}\label{thm_oregular}
Let $1\leq m\leq n$. Let $X$ be a compact set in $\C^n$. Then the following conditions are equivalent:
\begin{enumerate}\itemsep2mm
\item $X$ is a $O^m$-regular set;

\item for every $f\in \mathcal{C}(O^m_X)$ there exists a function $u\in\mathcal{SH}_m(X)\cap\mathcal{C}(X)$ such that
$u=f$ on $O^m_X$;

\item for every $f\in \mathcal{C}(B^m_X)$ there exists a function $u\in\mathcal{SH}_m(X)\cap\mathcal{C}(X)$ such that
$u=f$ on $B^m_X$.

\end{enumerate}
\end{theorem}
\begin{proof}
The implication (1)$\Rightarrow$(2) follows from Theorem 3.3 in~\cite{PoletskySigurdsson2}. To prove the implication (2)$\Rightarrow$(3) note that if $f\in \mathcal C(B_X^m)$, then $f\in \mathcal C(O_X^m)$, and therefore there exists $u\in \mathcal {SH}_m(X)\cap \mathcal C(X)$ such that $u=f$ on $O_X^m$. Since both functions $u$ and $f$ are continuous we obtain that $u=f$ on $B_X^m$. To prove the last implication  (3)$\Rightarrow$(1) suppose that there exists $z_0\in B_X^m\setminus O_X^m$. By Lemma~\ref{lem_choquet} there exists a function $f\in \mathcal C(X)$ such that $\textbf {S}_f(z_0)<f(z_0)$. By assumption there exists a function $u\in\mathcal{SH}_m(X)\cap\mathcal{C}(X)$ such that $u=f$ on $B^m_X$. Then we get that
\[
\textbf {S}_f(z_0)\geq u(z_0)=f(z_0)\, ,
\]
and a contradiction is obtained.
\end{proof}

Next, we consider the case when $X$ is equal to its Choquet boundary.

\begin{theorem}\label{thm_whole}
Let $X$ be a compact set in $\C^n$, and $1\leq m\leq n$. The following conditions are then equivalent:
\begin{enumerate}\itemsep2mm
\item $O_X^m=X$;

\item $B_X^m=X$;

\item $\mathcal C(X)=\mathcal {SH}_m(X)$;

\item $\mathcal C(X)=\mathcal {H}_m(X)$;

\item $f(z)=\|z\|^2\in \mathcal {H}_m(X)$;

\item $g(z)=-\|z\|^2\in \mathcal {SH}_m(X)$;

\item $\mathcal I_X^m(z)=\{z\}$ for all $z\in X$.
\end{enumerate}
\end{theorem}
\begin{proof} The following implications are obvious: (1)$\Rightarrow$(2), (3)$\Rightarrow$(4), (4)$\Rightarrow$(5), and (3)$\Rightarrow$(6).  We have that
(2)$\Rightarrow$(3) follows from Theorem~\ref{thm_oregular}. For implication (5)$\Rightarrow$(1) take $z_0\in X$. Since $-\|z\|^2\in \mathcal H_m(X)$, then also $-\|z-z_0\|^2$ is $m$-harmonic and it is also a peak function for $z_0$. Thus, $z_0\in O_X^m$. To note implication (6)$\Rightarrow$(5): Since $\|z\|^2$ is $m$-subharmonic, and by by assumption $-\|z\|^2$ is also $m$-subharmonic we have that $\|z\|^2$ is $m$-harmonic. Finally, the equivalence between (1) and (7) follows from Proposition~\ref{propositionIset}.
\end{proof}

\section{The Dirichlet problem for $m$-harmonic functions}\label{sec_DPharmonic}

In this section we shall characterize those compact sets for which the Dirichlet problem
has a solution for $m$-harmonic functions (Theorem~\ref{thm_poisson}). First let us compare
$m$-harmonic functions defined on a compact set with $m$-harmonic functions defined on an open set.

It was proved in~\cite{ACH} that every $m$-harmonic function defined on an \emph{open} set is pluriharmonic. The situation is different for the function theory on compact sets. We give in Example~\ref{ex_ph} an example of a $2$-harmonic function defined on a compact set that is not pluriharmonic ($3$-harmonic). On the other hand, in Proposition~\ref{prop_Bn} we show that there are compact sets $X$ for which $\mathcal {H}_m(X)=\mathcal {PH}(X)$.

\begin{example}\label{ex_ph}
Let $X=\{(0,0,z_3)\in \mathbb C^3: |z_3|\leq 1\}$, and let $u$ be a function defined on $X$ by $u(z_1,z_2,z_3)=-|z_3|^2$. Then  $-u$ is plurisubharmonic, and also $2$-subharmonic. Furthermore, $u$ is the restriction of a $2$-subharmonic function defined in $\mathbb C^3$; namely
\[
u(z_1,z_2,z_3)=2\left (|z_1|^2+|z_2|^2-\frac 12|z_3|^2\right),\qquad (z_1,z_2,z_3)\in X.
\]
Finally, note that $u$ is not plurisubharmonic ($3$-subharmonic) on $X$. To prove this assume by contradiction that $u\in \mathcal {SH}_3(X)$. By assumption there exists a decreasing sequence $u_j\in \mathcal {SH}_3^o(X)$ such that $u_j\to u$, as $j\to \infty$. But then $u_j$ must be subharmonic on the set $Y=X\cap \{|z_3|<1\}$, and therefore $u$ must be also subharmonic on $Y$, and a contradiction is obtained.\hfill{$\Box$}
\end{example}

\begin{proposition}\label{prop_Bn}
Let $\Omega\subset\mathbb C^n$ be a bounded $B$-regular domain in the sense of Sibony~\cite{sibony}. Then we have that $\mathcal {H}_m(\bar{\Omega})=\mathcal {PH}(\bar{\Omega})$.
\end{proposition}
\begin{proof}
Recall that if $\Omega$ is a $B$-regular domain, then for all $z\in \partial \Omega$ we have that $\mathcal J_z^n(\bar{\Omega})=\{\delta_z\}$. Take any $h\in \mathcal {H}_m(\bar{\Omega})$, then $h\in \mathcal {H}_m(\Omega)$, so $h\in \mathcal {PH}(\Omega)$. By the assumption of $B_n$-regularity we have also that $h\in \mathcal {PH}(\partial \Omega)$, which implies that $h\in \mathcal {PH}(\bar{\Omega})$.
\end{proof}

One of the main notions in Theorem~\ref{thm_poisson} is so called $m$-Poisson sets defined as follows.

\begin{definition}\label{def_poissonset}
Let $1\leq m\leq n$. A compact set $X$ in $\C^n$ is called a \emph{$m$-Poisson set} if for every $f\in \mathcal{C}(B^m_X)$, there exists a function $u\in\mathcal{H}_m(X)$ such that $u=f$ on $B^m_X$.
\end{definition}

In Example~\ref{ex1_poisson} we see that the topological closure of the unit ball in $\C^n$ is
a $O^m$-regular set, but not a $m$-Poisson set. Furthermore, in Example~\ref{ex2_poisson} we see that the topological closure of the unit polydisc in $\C^n$, $n\geq 3$, is a $n$-Poisson set.

\begin{example}\label{ex1_poisson}
Let $X$ be the topological closure of the unit ball $\mathbb{B}$ in $\C^n$, and let $1\leq m\leq n$. Then we have that
\[
O^m_X=B^m_X=\partial \mathbb{B}\, .
\]
Hence, $X$ is a $O^m$-regular set. But $X$ is not a $m$-Poisson set, since it is not always possible to extend a function $f\in \mathcal{C}(\partial \mathbb{B})$ to the inside so that it is $m$-pluriharmonic (see e.g.~\cite{Bedford}).\hfill{$\Box$}
\end{example}

\begin{example}\label{ex2_poisson}
Let $X$ be the closure of the unit polydisc $\mathbb{D}^n$ in $\C^n$, and let $1\leq m\leq n$. Then
\begin{align*}
O^1_X=B^1_X & =\partial {\bar{\mathbb{D}}}^n \, , \\
O^n_X=B^n_X & =\partial \mathbb{D}\times\dots\times \partial \mathbb{D}\, ,
\end{align*}
and for $1<m<n$ we get that
\[
O^m_X=B^m_X=\bigcup_{1\leq j_1<\dots<j_m\leq n}\partial \mathbb{D}\times\dots\times\overbrace{\bar{\mathbb {D}}}^{j_1}\times \dots \times\overbrace{\bar{\mathbb {D}}}^{j_m}\times \dots\times \partial \mathbb{D}.
\]
Thus, $O^n_X$, and  $B^n_X$ are equal to the distinguished boundary of $\mathbb{D}^n$. For $1<m<n$, the above statement follows from the fact that any $m$-subharmonic function in $\Omega\subset \mathbb C^n$ is $(m-1)$-subharmonic on any hyperplane passing by $\Omega$ (see~\cite{AS}). In particular, $X$ is a $O^m$-regular set. Furthermore, for $n\geq 3$ the compact set $X$ is a $n$-Poisson set, since for every $f\in \mathcal{C}(\partial \mathbb{D}\times\dots\times \partial \mathbb{D})$ we can
always find a pluriharmonic function $u$ defined on $\mathbb{D}^n$ such that $u=f$ on $\partial \mathbb{D}\times\dots\times \partial \mathbb{D}$ (see e.g.~\cite{ACpluriharmonic, ACH}).\hfill{$\Box$}
\end{example}

Let us next define an (partial) order in the cone of Jensen measures.

\begin{definition}
Let $\mu$ and $\nu$ be Jensen measures. We say that $\mu$ is \emph{subordinated} to $\nu$, and denote it with $\mu\preccurlyeq \nu$, if for all $u\in \mathcal {SH}_m(X)\cap \mathcal C(X)$ it holds that
\[
\int_Xu\,d\mu\leq \int_Xu\,d\nu\, .
\]
\end{definition}
\begin{remark}
Note that $\preccurlyeq$ is indeed an (partial) order, see e.g.~\cite{Gamelin}.
\end{remark}

Lemma~\ref{lem_poisson} shall later be used in Proposition~\ref{prop_PBf} and in Theorem~\ref{thm_poisson}. The set $\mathcal{J}_z^{b,m}$ of Jensen measures used in Lemma~\ref{lem_poisson} was defined in Definition~\ref{def_JzBm}.

\begin{lemma}\label{lem_poisson}
Let $1\leq m\leq n$, and let $X$ be a compact set in $\C^n$. For every $\mu\in \mathcal{J}_z^m$ there exists a measure $\nu\in \mathcal{J}_z^{b,m}$ with $\mu\preccurlyeq \nu$.
\end{lemma}
\begin{proof} Cf. Theorem~1.17 in~\cite{Gamelin}.
\end{proof}

In contrast with Definition~\ref{def} we shall in Definition~\ref{def_PBf} introduce an envelope
construction where the given function is only defined on the \v{S}ilov boundary, $B^m_X$. We
name this envelope as customary after Oskar Perron and Hans-Joachim Bremermann. In Proposition~\ref{prop_PBf},
we prove some elementary, but useful facts about this envelope.

\begin{definition}\label{def_PBf}
Let $1\leq m\leq n$, and $X$ be a compact set in $\C^n$. For given $f\in \mathcal{C}(B^m_X)$ we define the \emph{Perron-Bremermann envelope}, $PB_f$, as
\[
PB_f=\sup\left\{v(z): v\in \mathcal {SH}_m(X)\cap\mathcal{C}(X), v\leq f\, \text { on }\, B^m_X\right\}\, .
\]
\end{definition}

\begin{proposition}\label{prop_PBf}
Let $1\leq m\leq n$, and $X$ be a compact set in $\C^n$. For every $f\in \mathcal{C}(B^m_X)$ we have that
\begin{enumerate}\itemsep2mm
\item $PB_f$ is a lower semicontinuous function, and that for any $z\in X$ and for any $\mu \in \mathcal J_z^m$ it holds that
\[
PB_f(z)\leq \int PB_f\,d\mu \ ;
\]

\item if $X$ is a $O^m$-regular set, then
\[
PB_f(z)=\inf \left \{\int f\,d\mu; \mu\in \mathcal J_z^{b,m} \right\}\, .
\]
\end{enumerate}
\end{proposition}
\begin{proof} Part (1) is an immediate consequence of the definition. Next, to part (2). For $f\in \mathcal C(B_X^m)$,  let
\[
I_f(z)=\inf \left \{\int f\,d\mu; \mu\in \mathcal J_z^{b,m} \right\}\, .
\]
By construction we have that if $v\in \mathcal {SH}_m(X)\cap \mathcal C(X)$, and $v\leq f$ on $B_X^m$, then for any $\mu\in \mathcal J_z^{b,m}$ it holds that
\[
v(z)\leq \int v\,d\mu\leq \int f\,d\mu\, ,
\]
and therefore $PB_f(z)\leq I_f(z)$. By Theorem~\ref{thm_oregular}, there exists a function $u$ such that $-u\in \mathcal {SH}_m(X)\cap \mathcal C(X)$ and $u=f$ on $B_X^m$. Thanks to Lemma~\ref{lem_poisson} it holds that for every $\mu\in \mathcal J_z^m$ there is a Jensen measure $\nu\in \mathcal J_z^{b,m}$ such that
\[
\int u\,d\mu\geq \int u\, d\nu\, .
\]
Thus,
\[
PB_f(z)\geq \textbf{S}_u(z)=\inf\left \{\int u\,d\mu; \mu\in \mathcal J_z^m\right\}=I_u(z)=I_f(z)\geq PB_f(z)\, .
\]

\end{proof}

We end this note with Theorem~\ref{thm_poisson}, and the characterization of those compact sets $X$ for which the Dirichlet problem has a solution within the class of $m$-harmonic functions defined on a compact set.

\begin{theorem}\label{thm_poisson}
Let $X$ be a compact set in $\C^n$, and let $1\leq m\leq n$. Then the following conditions are equivalent:
\begin{enumerate}\itemsep2mm
\item $X$ is a $m$-Poisson set;

\item for every $z\in X$, the set $\mathcal{J}_z^{b,m}$ contains exactly one measure, $P_z^m$;

\item for every $f\in \mathcal{C}(B^m_X)$ we have that
\[
PB_{-f}=-PB_f\quad \text{ on } X\, .
\]
\end{enumerate}
\end{theorem}
\begin{proof} To prove the implication (1)$\Rightarrow$(2) assume that $X$ is a $m$-Poisson set, $z\in X$, and $\mu,\nu\in \mathcal J_z^{b,m}$. By the assumptions we have that for any $f\in \mathcal C(B_X^m)$ there exists a function $h\in \mathcal H_m(X)$ with $h=f$ on $B_X^m$. Hence,
\[
u(z)=\int f\,d\mu=\int f\,d\nu\, ,
\]
which implies that $\mu=\nu$. Next, we shall verify the implication (2)$\Rightarrow$(3). First we shall prove that $X$ is an $O^m$-regular set. Let $O_X^m\ni z_j\to z\in X$, as $j\to \infty$. Then from the sequence of measures $\delta_{z_j}$ we can extract a subsequence (denoted also by $\delta_{z_j}$) such that $\delta_{z_j}$ is weak$^*$-convergent to some measure $\mu\in \mathcal J_z^{b,m}$. By assumption the measure $\mu=P_z^m$ is unique. Then for every $f\in \mathcal C(X)$ we have that
\[
\int f\,dP_z^m=\lim_{j\to \infty}\int f\,d\delta_{z_j}=\lim_{j\to \infty}f(z_j)=f(z)=\int f\,d\delta_z\, ,
\]
which means that $P_z^m=\delta_z$, so $z\in O_X^m$. Thus, $X$ is an $O^m$-regular set. Now for $f\in\mathcal C(B_X^m)$ let us define the following function
\[
h(z)=\int f\,dP_z^m\, .
\]
We are going to prove that $h$ is a $m$-harmonic function. First we show that $h$ is continuous. Let $X\ni z_j\to z\in X$. We can assume that $P_{z_j}^m$ is weak$^*$-convergent to some measure $\mu\in \mathcal {J}_z^{b,m}$, if necessary we extract a subsequence. Therefore, by assumption $\mu=P_z^m$, and it follows that
\[
\lim_{j\to \infty}h(z_j)=\lim_{j\to \infty}\int f\,dP_{z_j}^m=\int f\,dP_z^m=h(z)\, .
\]
Proposition~\ref{prop_PBf} part (2), gives us that for every $z\in X$ it holds that
\begin{equation}\label{1}
h(z)=\int f\,dP_z^m=\inf \left \{\int f\,d\mu; \mu\in \mathcal J_z^{b,m} \right\}=PB_f(z)\, ,
\end{equation}
and therefore again by Proposition~\ref{prop_PBf} part (1) we can conclude that $h\in \mathcal {SH}_m(X)\cap \mathcal C(X)$. On the other hand,
for any $z\in X$ and any $\mu\in \mathcal J^m_z$ we have by Lemma~\ref{lem_poisson} that
\[
h(z)\leq \int fd\mu\leq \int h\,dP_z^m=\int f\,dP_z^m=h(z)\, .
\]
Hence, $h\in \mathcal H_m(X)$. Note also that it follows from (\ref{1}) that
\[
PB_{-f}(z)=\int -f\,dP_z^m=-\int f\,dP_z^m=-PB_f(z)\, .
\]
Next we shall prove the implication (3)$\Rightarrow$(1). By Proposition~\ref{prop_PBf} part (1), the envelopes $PB_f$ and $PB_{-f}$ are lower semicontinuous and therefore we can conclude that $PB_f\in \mathcal C(X)$. Furthermore, we have that for any $z\in X$ and any $\mu \in \mathcal J_z^m$ it holds that
\[
-PB_f(z)=PB_{-f}(z)\leq \int PB_{-f}\, d\mu=\int -PB_f\, d\mu\leq -PB_f(z)\, .
\]
Hence, $PB_f\in \mathcal H_m(X)$. To conclude that $(1)$ holds note that $PB_f\leq f$, and $PB_{-f}\leq -f$, on $B_X^m$ and therefore we must have that $PB_f=f$ on $B_X^m$.
\end{proof}

\begin{corollary}\label{cor_poisson}
Let $1\leq m\leq n$. Then every compact set $X$ in $\C^n$ that is a $m$-Poisson set is also an $O^m$-regular set.
\end{corollary}
\begin{proof} This follows immediately from Theorem~\ref{thm_oregular} and Theorem~\ref{thm_poisson}.
\end{proof}


\begin{thebibliography}{99}

\bibitem{AS}  Abdullaev B. I., Sadullaev A., Potential theory in the class of $m$-subharmonic functions. Proc. Steklov Inst. Math.  279  (2012),  no. 1, 155-180.

\bibitem{ACpluriharmonic} \AA hag P, Czy\.z R., Continuous pluriharmonic boundary values. Ann. Polon. Math. 91 (2007), no. 2-3, 99-117.

\bibitem{ACH} \AA hag P, Czy\.z R., Hed L., Extension and approximation of $m$-subharmonic functions. Complex Var. Elliptic Equ. 63 (2018), no. 6, 783-801.

\bibitem{ACH2} \AA hag P., Czy\.z R., Hed L., The geometry of $m$-hyperconvex domains. J. Geom. Anal 28 (2018), 3196-3222.


\bibitem{Bedford} Bedford E., The Dirichlet problem for some overdetermined systems on the unit ball in $\C^n$. Pacific J. Math. 51 (1974), 19-25.

\bibitem{BlietnerHansen} Bliedtner J., Hansen W., Potential theory. An analytic and probabilistic approach to balayage. Universitext. Springer-Verlag, Berlin, 1986.

\bibitem{Gamelin}  Gamelin T. W., Uniform algebras and Jensen measures. London Mathematical Society Lecture Note Series, 32. Cambridge University Press, Cambridge-New York, 1978.

\bibitem{Hansen} Hansen W., Harmonic and superharmonic functions on compact sets. Illinois J. Math.  29  (1985),  no. 1, 103-107.

\bibitem{klimek} Klimek M. Pluripotential theory. Vol. 6, London Mathematical Society Monographs. New Series. New York (NY): Oxford Science Publications, The Clarendon Press, Oxford University Press; 1991.

\bibitem{L} Lu H.-C., Complex Hessian equations. Doctoral thesis, University of Toulouse III Paul Sabatier, 2012.

\bibitem{perkins} Perkins T. L., The Dirichlet problem for harmonic functions on compact sets. Pacific J. Math. 254 (2011), no. 1, 211-226.

\bibitem{Poletsky1} Poletsky E. A., Analytic geometry on compacta in $\mathbb{C}^n$. Math. Z. 222 (1996), no. 3, 407-424.

\bibitem{Poletsky2} Poletsky E. A., Approximation by harmonic functions. Trans. Amer. Math. Soc. 349 (1997), no. 11, 4415-4427.

\bibitem{PoletskySigurdsson1} Poletsky E. A., Sigurdsson R., Dirichlet problems for plurisubharmonic functions on compact sets. Manuscript (2010), arXiv:1005.0248

\bibitem{PoletskySigurdsson2} Poletsky E. A., Sigurdsson R., Dirichlet problems for plurisubharmonic functions on compact sets. Math. Z.  271  (2012),  no. 3-4, 877-892.

\bibitem{sibony} Sibony N., Une classe de domaines pseudoconvexes. Duke Math. J. 55 (1987), no. 2, 299-319.

\end{thebibliography}
\end{document}